\documentclass[letterpaper, 10 pt, conference]{ieeeconf}  

\IEEEoverridecommandlockouts                              

\usepackage{graphics} 
\usepackage{epsfig} 
\usepackage{amsmath} 
\usepackage{amssymb}  
\usepackage{amsthm}
\usepackage{enumerate}

\usepackage{arydshln}

\usepackage{caption}
\usepackage{float}
\floatstyle{boxed}
\restylefloat{figure}

\title{\LARGE \bf
$\hf$ control problem for general discrete--time systems*}

\author{Sebastian F. Tudor$^{1}$, Cristian Oar\ua$^{2}$ and  \c Serban Sab\ua u$^{1}$
\thanks{*This work was supported by a grant of the Romanian National Authority for Scientific Research, CNCS -- UEFISCDI, PN-II-ID-PCE-2011-3-0235.}
\thanks{$^{1}$Sebastian Florin Tudor and \c Serban Sab\ua u are with the Department of Electrical and Computer Engineering, Stevens Institute of Technology, Hoboken, NJ 07030, USA. The first author is on leave from "Politehnica" University of Bucharest, Romania. }%
\thanks{$^{2}$Cristian Oar\ua\ is with the Department of Automatic Control and Systems Engineering, Faculty of Automatic Control and Computer Science, "Politehnica" University of Bucharest, Romania.
        }%
\thanks{E-mails: {\tt\small studor@stevens.edu}, {\tt\small ssabau@stevens.edu}, {\tt\small cristian.oara@acse.pub.ro} (corresponding author).}%
}

\begin{document}

\theoremstyle{plain}
\newtheorem{thm}{Theorem}
\newtheorem{prop}[thm]{Proposition}
\newtheorem{cor}[thm]{Corollary}
\newtheorem{lem}[thm]{Lemma}
\newtheorem{df}[thm]{Definition}
\theoremstyle{definition}
\newtheorem{rem}[thm]{Remark}

\newcommand{\R}{\mathbb{R}} 
\newcommand{\C}{\mathbb{C}}
\newcommand{\Z}{\mathbb{Z}}
\newcommand{\F}{\mathbb{F}}
\newcommand{\N}{\mathbb{N}}
\newcommand{\Di}{\mathbb{D}}
\newcommand{\pDi}{\partial\mathbb{D}}

\newcommand{\w}{\omega}
\newcommand{\D}{\Delta}
\newcommand{\x}{\times }
\newcommand{\jw}{ j\omega }
\newcommand{\ejt}{ e^{j\theta} }
\newcommand{\sg}{\sqrt\gamma}

\newcommand{\Sig}{\Sigma}
\newcommand{\si}{\sigma}
\newcommand{\al}{\alpha}
\newcommand{\ab}{\bar\alpha}
\newcommand{\ap}{\alpha^2}
\newcommand{\Lam}{\Lambda}
\newcommand{\T}{\textbf{T}}
\newcommand{\G}{\textbf{G}}
\newcommand{\K}{\textbf{K}}
\newcommand{\Q}{\textbf{Q}}
\newcommand{\fP}{{\bf\Pi_\Sigma}}
\newcommand{\zet}{z}

\newcommand{\ai}{\^{a}}
\newcommand{\tz}{\c{t}}
\newcommand{\sh}{\c{s}}
\newcommand{\ii}{\^{\i}}
\newcommand{\II}{\^{I}}
\newcommand{\ua}{\u{a}}
\newcommand{\Sh}{\c{S}}

\newcommand{\ba}{\left[ \begin{array}}
\newcommand{\baa}{\begin{array}}
\newcommand{\ea}{\end{array} \right]}
\newcommand{\eaa}{\end{array}}

\newcommand{\be}{\begin{equation}}
\newcommand{\ee}{\end{equation}}
\newcommand{\bel}{\begin{equation}\label}
\newcommand{\rf}[1]{(\ref{#1})}

\newcommand{\ov}{\overline}
\renewcommand{\tilde}{\widetilde}
\newcommand{\ti}{\widetilde}
\newcommand{\wtilde}{\widetilde}
\newcommand{\what}{\widehat}
\renewcommand{\hat}{\widehat}
\renewcommand{\bar}{\ov}

\newcommand{\dd}{{\mathrm{d}}}
\renewcommand{\d}{{\mathrm{d}}}
\newcommand{\trace}{{\mathrm{Trace}}\,}
\newcommand{\Sum}{{\sum}}
\newcommand{\no}{{\mathrm{no}}\,}
\newcommand{\Ker}{{\mathrm{Ker}}\,}
\renewcommand{\ker}{\Ker}
\newcommand{\sgn}{{\mathrm{sgn}}\,}
\newcommand{\diag}{{\mathrm{diag}}\,}
\renewcommand{\Im}{{\mathrm{Im}}\,}
\newcommand{\Ricc}{Ricc}
\newcommand{\rank}{{\mathrm{rank}}\,}
\newcommand{\LFT}{{\mathrm{LFT}}}
\newcommand{\TLFS}{{\mathrm{TLFS}}}
\newcommand{\TLFI}{{\mathrm{TLFI}}}

\newcommand{\calRLinf}{{\cal R}{\cal L}_{\infty}}
\newcommand{\CC}{{{\mbox{\rm \hspace*{0.05ex}
\rule[.18ex]{.18ex}{1.24ex} \kern -.65em C}}}}
\newcommand{\rhinf}{{\cal R \cal H}^\infty}
\newcommand{\hinf}{{\cal  \cal H}^\infty}

\newcommand{\Le}{{{\cal  L}}_e}
\newcommand{\script}[1]{\EuScript{#1}}
\newcommand{\range}{\operatorname{Range}}
\newcommand{\proj}{\operatorname{proj}}
\newcommand{\jac}{\operatorname{Jac}}
\newcommand{\zip}{\operatorname{zip}}
\newcommand{\CRICOP}{{\mathrm{CRICOP}}}
\newcommand{\DRICOP}{{\mathrm{DRICOP}}}
\newcommand{\inertia}{\operatorname{In\,}}
\newcommand{\Gl}{\operatorname{Gl\,}}
\newcommand{\real}{\operatorname{Re}}
\newcommand{\col}{\operatorname{col}}


\newcommand{\rL}{RL^\infty}
\newcommand{\rHp}[1]{{RH^\infty_{+ #1}}}
\newcommand{\rHm}[1]{{RH^\infty_{- #1}}}

\newcommand{\hf}{\mathcal{H}_\infty}
\newcommand{\rhf}{\mathcal{RH}_\infty}
\newcommand{\bhf}{\mathcal{BH}_\infty^{(\gamma)}}
\newcommand{\bhfu}{\mathcal{BH}_\infty^{(1)}}
\newcommand{\hdoi}{\mathcal{H}_2}

\newenvironment{matr}[1]{\left[ \begin{array}{#1}}{\end{array} \right]}
\renewenvironment{proof}{{\noindent \bf Proof.}}
{\hfill \mbox{}\hfill  \rule{2.5mm}{2.5mm}\\[+1mm]}
\def\Re{\mathop{\mathrm{Re}}}
\def\Im{\mathop{\mathrm{Im}}}
\def\esssup{\mathop{\mathrm ess sup}}
\maketitle
\thispagestyle{empty}
\pagestyle{empty}

\begin{abstract}
This paper considers the $\hf$ control problem for a general discrete--time system, possibly improper or polynomial. The parametrization of suboptimal $\hf$ output feedback controllers is presented in a realization--based setting, and it is given in terms of two descriptor Riccati equations. Moreover, the solution features the same elegant simplicity of the proper case. An interesting numerical example is also included.


\end{abstract}

\section{INTRODUCTION}

Ever since it emerged in the 1980's in the seminal paper of Zames \cite{zam81}, the $\hf$ control problem (also known as the disturbance attenuation problem) has drawn much attention, mainly due to the wide range of control applications. It is one of the most celebrated problems in the control literature, since it can be approached from diverse technical backgrounds, each providing its own interpretation.  

The design problem is concerned with finding the class of controllers, for a given system, that stabilizes the closed--loop system and makes its input--output $\hf$--norm bounded by a prescribed threshold. Various mathematical techniques were used, e.g., Youla parametrization, Riccati--based approach, linear matrix inequalities, to name just a few.   

The original solution involved analytic functions (NP interpolation) or operator theory \cite{sar67,aak71}. For good surveys on the classical topics we refer to \cite{fd87,fra87}. Notable contributions to the state--space solution for the $\hf$ control problem are due to \cite{dgkf,tad90,sto92}. An algebraic technique using a chain scattering approach is presented in \cite{kimura}. The solution of the $\hf$ control problem in discrete--time setting is given in \cite{IW93}. 

More recently, $\hf$ controllers for general continuous--time systems (possibly improper or polynomial) were obtained. An extended model matching technique was employed in \cite{takaba94}. A solution expressed in terms of two generalized algebraic Riccati equations is given in \cite{wang}. A matrix inequality approach was considered in \cite{lmi97}. Note that a {\it dicrete--time} solution is still missing.  

General systems cover a wide class of physical systems, e.g. non--dynamic algebraic constraints (differential--algebraical systems), impulsive behavior in circuits with inconsistent initial conditions \cite{tolsa}, and hysteresis, to name just a few. Cyber--physical systems under attack, mass/gas transportation networks, power systems and advanced communication systems can also be modeled as improper systems \cite{bullo13}. The wide range of applications of improper systems spans topics from engineering, e.g. aerospace industry, robots, path prescribed control, mechanical multi--body systems, network theory \cite{gunther99,rabier00,Lind2002765}, to economics \cite{lue77}.

Motivated by this wide applicability and interest shown in the literature for improper systems, we extend in this paper the $\hf$ control theory for general {\it discrete--time systems} using a novel approach, based on Popov's theory \cite{popov} and on the results in \cite{genricc}. A realization--based solution is provided, using a novel type of algebraic Riccati equation, investigated in \cite{oara2013ricc}. Our solution exhibits a numerical easiness similar with the proper case and can be seen as a straightforward generalization of \cite{dgkf}.

The paper is organized as follows. In Section II we give some preliminary results. In Section III we state the suboptimal $\hf$ output feedback control problem. We provide in Section IV the main result, namely realzation--based formulas for the class of all stabilizing and contracting controllers for a general discrete--time transfer function matrix (TFM). In order to show the applicability of our results, we present in Section V an interesting numerical example. The paper ends with several conclusions. We defer all the proofs to the Appendix.


\section{PRELIMINARIES}
We denote by $\C,\ \Di$, and $\partial\Di$ the complex plane, the open unit disk, and the unit circle, respectively. Let  $\bar\C=\C\cup\{\infty\}$ be the one--point compactification of the complex plane. Let $z\in\C$ be a complex variable. $A^*$ stands for the conjugate transpose of a complex matrix $A\in\C^{m\times n}$;  $A^{-1}$ denotes the inverse of $A$, and $A^{1/2}$ is such that $A^{1/2}A^{1/2}=A$, for $A$ square. The union of generalized eigenvalues (finite and infinite, multiplicities counting) of the matrix pencil $A-zE$ is denoted with $\Lam(A-zE)$, where $A,E\in\C^{n\times n}$. By $\C^{p\times m}(z)$ we denote the set of $p\x m$ TFMs with complex coefficients. $\rhf$ stands for the set of TFMs analytic in $\bar\C\backslash\Di.$ The Redheffer product is denoted with $\otimes$.

To represent an improper or polynomial discrete--time system $\G\in\C^{p\x m}(z)$, we will use a general type of realization called {\it centered}:
\bel{realc}
\G(z)=D+C(z E-A)^{-1}B(\al- \beta z)
 =: \ba{c|c}
A-z E & B\\
\hline
C & D\ea_{z_0},
\ee 
where $z_0=\al/\beta\in\C$ is fixed, $n$ is called the order (or the dimension) of the realization, $A,E\in\C^{n\x n}$, $B\in\C^{n\x m}$, $C\in\C^{p\x n}$, $D\in\C^{p\x m}$, $\rank E\le n$, and the matrix pencil $A-z E$ is regular, i.e., $\det(A-zE)\not\equiv0$. Note that for $\al=1$ and $\beta=0$ we recover the well--known descriptor realization \cite{dai89} for an improper system, centered at $z_0=\infty$. We call the realization \rf{realc} {\it minimal} if its order is as small as possible among all realizations of this type.

Centered realizations have some nice properties, due to the flexibility in choosing $z_0$ always disjoint from the set of poles of $\G$, e.g., the order of a centered minimal realization always equals the McMillan degree $\delta(\G)$ and $\G(z_0)$ equals the matrix $D$ in \eqref{realc}. We call the realization \eqref{realc} {\it proper} if $\al E - \beta A$ is invertible. Thus, by using centered realizations we recover standard-like characterization of the TFM. Centered realizations have been widely used in the literature to solve problems for generalized systems whose TFM is improper \cite{gohberg88,gohberg92,rakowski92,oara11}. Throughout this paper, we will consider proper realizations centered on the unit circle, i.e.,  $z_0\in\pDi$ not a pole of $\G$. Furthermore, we consider $\al\in\pDi,\ \beta:=\ab$, and thus $z_0=\al/\ab=\ap\in\partial\Di$.  

Conversions between descriptor realizations and centered realizations on $z_0\in\partial\Di$ can be done can be done by simple manipulations. Consider a descriptor realization
\be
\G(z) = D+C(z E-A)^{-1}B =:\ba{c|c}
A-z E & B\\
\hline
C & D\ea_{\infty}
\ee
and fix $z_0\in\partial\Di$. Then there exist $U$ and $V$ two invertible (even unitary) matrices such that
\bel{dec}
U(A-zE)V = \ba{cc}
A_1 - zE_1 & A_{12} - zE_{12}\\
0 & A_2
\ea,
\ee
where $A_2$ is nonsingular (contains the non--dynamic modes) and $\rank\ba{cc}E_1&E_{12}\ea = \rank E$, see \cite{oara09} for proof and numerical algorithms. Let 
$$\ba{c} B_1 \\ B_2\ea:=V^*(A-z_0 E)^{-1}B,\ \ba{cc} C_1 & C_2\ea :=CV,$$
where the partitions are conformable with \eqref{dec}. A direct check shows that the following realization of $\G$ is centered at $z_0$ and proper:
\be
\G(z) = \ba{c|c}
A_1 - zE_1 & -E_1B_1-E_{12}B_2\\
\hline
C_1 & D-C_1B_1-C_2B_2
\ea_{z_0}.
\ee

We say that the system \eqref{realc} is stable if its {\it pole pencil} $A-zE$ has $\Lambda(A-zE)\subset\Di$, see e.g. \cite{dai89}. Note that any stable system belongs to $\rhf$. The {\it system pencil} is by definition
$$\mathcal{S}_\G(z):=\ba{cc}
A-zE & B(\al-\beta z)\\
C & D
\ea.
$$

The pair $(A-zE,B)$ is called {\it stabilizable} if (i) $\rank \ba{cc}A-z E & B\ea=n,$ for all $z\in\C\backslash\Di$, and (ii) $\rank \ba{cc} E & B\ea=n$. We call the pair $(C,A-zE)$ detecta-ble if the pair $(A^*-zE^*,C^*)$ is stabilizable.

We say that a square system $\G\in\C^{m\x m}(z)$ is {\it unitary} on the unit circle if $\G^\#(z)\G(z)=I, \forall z\in\pDi\backslash\Lam(A-zE)$, where $\G^\#(z):=\G^*(1/z^*)$. If, in addition, $\G\in\rhf$ then $\G$ is called {\it inner}. The following lemma will be used in the sequel to characterize inner systems given by centered realizations (see for example \cite{goh92} and \cite{oara99}).  

\begin{lem}
\label{inner_lem}
Let $\G$ be a TFM without poles at $z_0$, having a minimal realization as in \eqref{realc}. Then $\G$ is unitary (inner) iff $D^*D=I_m$ and there is an invertible (negative definite) Hermitian matrix $X=X^*$ such that 
\bel{inner_e}
\baa{lcr}
E^*XE-A^*XA+C^*C &=&0,\\
D^*C+B^*X(\al E-\beta A) &=&0.
\eaa
\ee
\end{lem}

Let $\G\in\mathcal{RL}_\infty(\pDi)$, the Banach space of general discrete--time TFMs (possibly improper or polynomial) that are bounded on $\pDi$. Then the $\hf$--norm of $\G$ is defined as:
$$\|\G\|_\infty:=\sup_{\theta\in[0,2\pi)} \sigma_{max}\big(\G(\ejt)\big).$$
We denote by $\bhf$ the set of all stable and bounded TFMs, that is, $\bhf:=\{\G\in\rhf : \|\G\|_\infty<1\}.$

Consider now the structure $\Sig:=(A-zE,B;\ Q,L,R)$, where $A,E\in\C^{n\x n},B,L\in\C^{m\x n}, Q=Q^*\in\C^{n\x n},R=R^*\in\C^{m\x m}$. $\Sig$ can be seen as an abbreviated representation of a controlled system $\G$ and a quadratic performance index, see \cite{genricc,eu13diss}. We associate with $\Sig$ two mathematical objects of interest. The matrix equation 
\bel{ricc}
\baa{r}
E^*XE-A^*XA+Q-\big((\al E-\beta A)^*XB+L\big)\cdot\\\cdot R^{-1}\big(L^*+B^*X(\al E-\beta A)\big)=0
\eaa
\ee
is called {\it the descriptor discrete--time algebraic Riccati equation} and it is denoted with DDTARE$(\Sig)$. Necessary and sufficient existence conditions together with computable formulas are given in \cite{oara2013ricc}. We say that the Hermitian square matrix $X=X^*\in\C^{n\x n}$ is the {\it unique} stabilizing solution to DDTARE($\Sig$) if $\Lam\big(A-zE + BF(\al-\beta z)\big)\subset\Di$, where 
\bel{stabfeed}
F:=-R^{1}\big(B^*X(\al E -\beta A) + L^*\big)
\ee
is the stabilizing feedback. We define next a parahermitian TFM $\fP\in\C^{m\x m}(z)$, also known as the discrete--time Popov function \cite{genricc}:
\bel{popovf}
{\bf\Pi_\Sig}(z) = \ba{cc|c}
A-zE & 0 & B\\
Q(\al-\beta z) & E^*-zA^* & L\\
\hline
L^* & B^* & R
\ea_{z_0}.
\ee
It can be easily checked that ${\bf\Pi_\Sig}$ is exactly the TFM of the Hamiltonian system, see \cite{eu13diss}. Moreover, the descriptor symplectic pencil, as defined in \cite{oara2013ricc}, is exactly the system pencil $\mathcal{S}_{\fP}$ associated with the realization \eqref{popovf} of ${\bf\Pi_\Sig}$. We are now ready to state two important results.  

\begin{prop}
\label{pp2}
Let $\Sig:=(A-zE,B;\ Q,L,R)$. Assume $\Lam(A-zE)\subset\Di$. The following statements are equivalent.

\begin{enumerate}[(i)]
\item $\fP(\ejt)<0,$  for all $\theta\in[0,2\pi)$.
\item $R<0$ and DDTARE($\Sig$) has a stabilizing hermitian solution $X=X^*$.
\end{enumerate}
\end{prop}

\begin{prop}
\label{BRL}
{\textbf{\emph (Bounded-Real Lemma)}} Let $\G\in\C^{p\x m}(z)$ having a minimal proper realization as in \eqref{realc} and consider $\tilde\Sig:=(A-zE,B;\ C^*C,C^*D, D^*D-I_m)$. Then the following statements are equivalent.
\begin{enumerate}[(i)]
\item $\G\in\bhf$, i.e., $\Lam(A-zE)\subset\Di$ and $\|\G\|_\infty<1$.
\item $D^*D-I_m<0$ and DDTARE($\tilde\Sig$) has a stabilizing hermitian solution $X=X^*\le0$.
\end{enumerate}
\end{prop}

\section{PROBLEM FORMULATION}

Let $\T\in\C^{p\x m}(z)$ be a general discrete--time system, possibly improper or polynomial, with input $u$ and output $y$, written in partitioned form:
\be
\ba{c} y_1 \\y_2 \ea = \T \ba{c} u_1\\u_2 \ea =
\ba{cc} \T_{11} & \T_{12}\\
\T_{21} & \T_{22}\ea \ba{c} u_1\\u_2 \ea,
\ee
where $\T_{ij}\in\C^{p_i\x m_j}(z)$ with $i,j\in \{1,2\}$, $m:=m_1+m_2$, $p:=p_1+p_2$. The {\it suboptimal $\hf$ control problem} consists in finding all controllers $\K\in\C^{m_2\x p_2}(z)$, $u_2=\K y_2$, for which the closed--loop system 
\bel{lft}
\G:= \textbf{LFT}(\T,\K) := \T_{11}+\T_{12}\K(I-\T_{22}\K)^{-1} \T_{21}
\ee
is well--posed, stable and $\|\G\|_\infty<1$, i.e., $\G\in\bhf$.

We make a set of additional assumptions on $\T$ which either simplify the formulas with no loss of generality, or are of technical nature. Let
\bel{sspace}
\T(z)=
\ba{c|cc}
A-z E &B_1& B_2\\
\hline
C_1 & 0 & D_{12}\\
C_2 & D_{21} & 0
\ea_{z_0}
\ee
be a minimal realization with $z_0\in\pDi\backslash\Lam(A-zE)$.
\begin{enumerate}[{\bf (H$_1$)}]
\item The pair $(A-zE,B_2)$ is stabilizable and the pair $(C_2,A-zE)$ is detectable.
\item For all $\theta\in [0,2\pi)$, we have that
\bel{zeros11}
\rank \ba{cc}
A-\ejt E & B_2 (\al- \beta e^{j\theta})\\
C_1 & D_{12}
\ea = n+m_2.\ee

\item For all $\theta\in [0,2\pi)$, we have that
\bel{zeros22}
\rank \ba{cc}
A-\ejt E & B_1 (\al- \beta e^{j\theta})\\
C_2 & D_{21}
\ea = n+p_2.\ee
\end{enumerate}

\begin{rem}
The hypothesis {\bf (H$_1$)} is a necessary condition for the existence of stabilizing controllers, see \cite{zhou} for the standard case. We assume in the sequel that {\bf (H$_1$)} is always fulfilled. 
\end{rem}

\begin{rem}
The hypotheses {\bf (H$_2$)} and {\bf (H$_3$)} are {\it regularity assumptions}, see \cite{zhou, genricc} for the standard case. In particular, it follows from {\bf (H$_2$)} that $\T_{12}$ has no zeros on the unit circle, $p_1\ge m_2$, and that $\rank D_{12}=m_2$ (thus $D_{12}^* D_{12}$ invertible). Dual conclusions follow from {\bf (H$_3$)}: $\T_{21}$ has no zeros on $\pDi$, $m_1\ge p_2$, $\rank D_{21}=p_2$, and $D_{21} D_{21}^*$ is invertible. 

Furthermore, we note that {\bf (H$_2$)} and {\bf (H$_3$)} are reminiscent from the general $\hdoi$ problem \cite{2013h2aut} and are by no means necessary conditions for the existence of a solution to the general $\hf$ control problem. If either of these two assumptions does not hold, we get a {\it singular} $\hf$ optimal control problem, which is beyond the scope of this paper. 
\end{rem}

\begin{rem}
We have implicitly assumed in \eqref{sspace} that $\T_{11}(z_0) = D_{11}=0$ and $\T_{22}(z_0) =D_{22}=0$, without restricting the generality. If $\K$ is a solution to the problem with $D_{22}=0$, then $\K(I+D_{22}\K)^{-1}$ is a solution to the original problem. The extension for $D_{11}\ne0$ follows by employing a technique similar to the one in Chapter 14.7, \cite{zhou}. In particular, it also follows from this assumption that the closed--loop system is automatically well--posed. 
\end{rem}

\section{MAIN RESULT}

\begin{figure*}
\caption{}
$$\Sig_c:=\left(A-zE, \ba{cc}B_1& B_2\ea;\ C_1^*C_1, \ba{cc}0& C_1^*D_{12}\ea, \ba{cc}
-I_{m_1}& 0\\
0 & D_{12}^*D_{12}\ea\right),$$

$$F_c:=\ba{c}F_1\\F_2\ea=
\ba{c}
B_1^*X(\al E-\beta A)\\
-(D_{12}^*D_{12})^{-1}\big(D_{12}^*C_1+B_2^*X(\al E-\beta A)\big)
\ea,
$$

$$\Sig_\x:=\left(A^*-zE^*+F_1^*B_1^*(\al-\beta z), 
\ba{c}-(D_{12}^*D_{12})^{1/2}F_2 \\ C_2+D_{21}F_1\ea^*;\ \ 
B_1B_1^*, \ba{cc}0& B_1D_{21}^*\ea, \ba{cc}
-I_{p_1}& 0\\
0 & D_{21}D_{21}^*\ea
    \right)$$
\label{mainres}
\end{figure*}


\begin{figure*}
\caption{}
\bel{mainctrl}
\textbf{C}(z) =
\ba{c|cc}
A-zE+(B F +B_ZC_F)(\al-\beta z) & B_Z & -B_2(D_{12}^*D_{12})^{-\frac{1}{2}}+(\al E-\beta A)ZF_2^*(D_{12}^*D_{12})^{\frac{1}{2}}\\
\hline
-F_2 & 0 & (D_{12}^*D_{12})^{-\frac{1}{2}}\\
(D_{21}D_{21}^*)^{-\frac{1}{2}}C_F & (D_{21}D_{21}^*)^{-\frac{1}{2}} & 0
\ea_{z_0},
\ee
$$\text{where } B:=\ba{cc} B_1& B_2\ea,\quad C_F:=C_2+D_{21}F_1,\quad B_Z: = -\big(B_1D_{21}^*+(\al E-\beta A)ZC_F^*\big)(D_{21}D_{21}^*)^{-1}.$$

\bel{centralctrl}
\K_0(z)=\ba{c|c}
A-zE+\big((B_1B_1^*X-B_2B_2^*X)(\al E-\beta A)-(\al E-\beta A)ZC_2^*C_2\big)(\al-\beta z) &
-(\al E-\beta A)ZC_2^*\\
\hline
B_2^*X(\al E-\beta A) & 0
\ea_{z_0}.
\ee
\end{figure*}

\begin{figure*}
\caption{}
\bel{exss}
\T(z) = \ba{cccr|r:r}
0.906488-z  & 0.0816012 &-0.0005 & 0				&-0.0015 & 0.0095 \\
    0.0741349&  0.90121-z & -0.000708383 & 0			& -0.0096 &0.0004\\
    0 & 0 & 0.132655-z & 0				 		&	0.8673 &0.0000\\
    0 & 0 & 0 & 1								&	1.0000 &-1.0000\\
    \hline
    1    & 0    & 0  &   1 &0 &-1\\
     0   &  1    & 0   & -1& 0 &1\\
     \hdashline
     0    & 0 &    1   & -5& 1& 0
  \ea_{z_0=1}, 
  \ee

  \bel{extf}
  \T(z) = \ba{cc}
 \dfrac{       z^4 - 2.939 z^3 + 2.989 z^2 - 1.158 z + 0.1078}
           { z^3 - 1.94 z^2 + 1.051 z - 0.1076}					&       \dfrac{-z^3 + 1.798 z^2 - 0.7928 z - 0.008547}{ z^2 - 1.808 z + 0.8109}\\\cr
 
 \dfrac{     -z^4 + 2.95 z^3 - 3.01 z^2 + 1.168 z - 0.1082}
           { z^3 - 1.94 z^2 + 1.051 z - 0.1076} 					&       \dfrac{z^3 - 1.808 z^2 + 0.8109 z + 0.0003578} {z^2 - 1.808 z + 0.8109}\\\cr
 
  \dfrac{    -5 z^2 + 5.796 z + 0.07141}
            {  z - 0.1327}									& 5z-5
 
\ea  ,
   \ee

  \bel{riccsol}
  X = \ba{rrrr}  
  -13.6023 & -13.7705 &   0.0187 &  -0.0025\\
  -13.7705 & -13.9409 &   0.0189 &  -0.0025\\
    0.0187   & 0.0189  & -0.0000  & -0.0000\\
   -0.0025  & -0.0025  & -0.0000 &  -0.0000\\
   \ea,\ 
    Z = \ba{rrrr}
     -0.0002   &-0.0004 &   0.0068  & -0.0081\\
   -0.0004   &-0.0007  &  0.0143   &-0.0171\\
    0.0068 &   0.0143 &  -0.2753  &  0.3297\\
   -0.0081 &  -0.0171   & 0.3297 &  -0.3948
   \ea,
   \ee
   
   \bel{ex_ctrl}
   \K(z)=\frac{-0.1561 z^4 + 0.459 z^3 - 0.467 z^2 + 0.1809 z - 0.01679}
 {   z^4 - 2.808 z^3 + 2.722 z^2 - 0.9979 z + 0.08391},
 \ee
 
 \bel{ex_cl}
\G(z) =  \dfrac{ \ba{c}
0.3 (z+0.0255) (z-0.1313) (z-0.8269) (z-0.9794) (z-0.9817) (z-1)\\\cr
    -0.2988 (z-0.01409) (z-0.1344) (z-0.8269) (z-0.9817) (z-0.984) (z-1)
    \ea}
          { (z+0.005487) (z-0.1327) (z-0.8269) (z-0.8685) (z-0.9817)^2}.
          \ee

\end{figure*}

The following theorem is a crucial result in $\hf$ control theory. In the literature, it is known as Redheffer theorem. 
\begin{thm}
\label{redh}
Assume that $\T$ in \eqref{sspace} is unitary, $D_{21}$ is square and invertible, $\Lam(A-zE-B_1D_{21}^{-1}C_2(\al - \beta z))\subset\Di$, and let $\K$ be a controller for $\T$. Then $\G\in\bhf$ if and only if $\T$ is inner and $\K\in\bhf$.
\end{thm}

Recall that we associate with $\Sig=(A-zE,B;\ Q,L,R)$ the DDTARE($\Sig$) in \eqref{ricc}. We are ready to state the main result.

\begin{thm}
\label{mainthm}
Let $\T\in\C^{p\x m}(z)$ having a minimal realization as in \eqref{sspace}. Assume that {\bf (H$_1$)}, {\bf (H$_2$)}, and {\bf (H$_3$)} hold. Supp-ose that DDTARE($\Sig_c$) and DDTARE($\Sig_\x$) have stabilizing solutions $X=X^*\le 0$ and $Z=Z^*\le0$, respectively, where $\Sig_c$ and $\Sig_\x$ are given in Box \ref{mainres}. Then there exists a controller $\K\in\C^{p_2\x m_2}(z)$ that solves the suboptimal $\hf$ control problem. Moreover, the set of all such $\K$ is given by
$${\normalfont\K=\textbf{LFT}(\textbf{C},\Q),}$$
where $\Q\in\bhf$ is an arbitrary stable and bounded parameter, and $\textbf{C}$ is given in \eqref{mainctrl}.
\end{thm}

Theorem \ref{mainthm} provides sufficient conditions for the existence of suboptimal $\hf$ controllers. Further, we can easily obtain {\it the central controller}, for which $\Q=0$, under the so called normalizing conditions: 
$$D_{12}^*\ba{cc} C_1 &D_{12}\ea = \ba{cc} 0 & I\ea,\ \ba{c} B_1\\ D_{21}\ea D_{21}^*=\ba{c} 0 \\ I\ea.$$

\begin{cor}
Take the same hypotheses as in Theorem \ref{mainthm}. Then the central controller under normalizing conditions is $\K_0(z)$ in \eqref{centralctrl}.
\end{cor}

\begin{rem}
Consider a proper system centered at $\infty$, for which $E=I_n$, $\alpha = 1$, and $\beta=0$. It can be easily checked that we recover the controller formulas from the standard case, see e.g. \cite{genricc} and \cite{zhou} for the continuous--time counterpart. 
\end{rem}

\section{A NUMERICAL EXAMPLE}

It is well--known that $\hf$ controllers are highly effective in designing robust feedback controllers with disturbance rejection for F--16 aircraft autopilot design. The discretized short period dynamics of the F--16 aircraft can be written as:
\bel{f16d}
x_{k+1}=\tilde Ax_k+\tilde B_1u_{1,k}+\tilde B_2u_{2,k},\ k\ge0,\ x_0=0.
\ee
The system has three states, and $m_1=m_2=1$. The discrete--time plant model, i.e., the matrices $\tilde A,\tilde B_1$, and $\tilde B_2$ in \eqref{f16d}, was obtained in \cite{AlTamimi2007473} with sampling time $T = 0.1$s. 

We consider here a {\it trajectory prescribed path control (TPPC) problem}. In general, a vehicle flying in space constrained by a set of path equations is modeled by a system of differential--algebraic equations, see e.g. \cite{1104236,tppc}. In order to obtain a TPPC problem, we add a pole at $\infty$ (a non--dynamic mode) by augmenting the system \eqref{f16d} as follows: 
\bel{aax}
A-zE=\ba{cc}\tilde A-zI_3&0\\0&1\ea,\ba{cc}B_1 & B_2\ea = \ba{cc} \tilde B_1 & \tilde B_2\\1 & -1\ea.
\ee

Assume that all the dynamical states are available for measurement. With this and the augmentation \eqref{aax}, we obtain a minimal realization with $z_0=1$ for the system $\T$, see \eqref{exss}, having $n=4,\ m_1=m_2=1,\ p_1=2,\ p_2=1$. The TFM of $\T$ is given in \eqref{extf}. Note that the system is improper, having one pole at $\infty$, and that $\delta(\T)=n$. 

For this system, we want to find a stabilizing and contracting controller using the formulas in Theorem \ref{mainthm}. 

It can be easily checked that the system $\T$ satisfies {\bf (H$_1$)}, {\bf (H$_2$)}, and {\bf (H$_3$)}. Furthermore, the DDTARE$(\Sig_c)$ and the DDTARE$(\Sig_\x)$ have stabilizing solutions $X=X^*\le 0$ and $Z=Z^*\le0$, given in \eqref{riccsol}. Moreover, the stabilizing feedback for $\Sig_c$ was computed to be: 
   $$F_c = \ba{rrrr}
   0.0031 &   0.0031&   -0.0000   & 0.0000\\
    0.5012&   -0.4988 &  -0.0000 &   1.0000\\
    \ea.$$

Therefore, $\T$ satisfies the conditions in Theorem \ref{mainthm}. Taking $\Q=0$, we obtain with Theorem \ref{mainthm} the central {\it proper} controller given in \eqref{ex_ctrl}. The closed--loop system $\G$ is given in \eqref{ex_cl}. Note that $\G$ is proper and stable, having the poles $\{0.0054,0.1327, 0.8269, 0.8685,0.9817\}\subset\Di$. Moreover, $\|\G\|_\infty = 0.4533<1$. The singular value plots of $\T$ and $\G$ are shown in Box \ref{sigma_plot}.

\begin{figure}[h]
\centering
\includegraphics[scale=0.5]{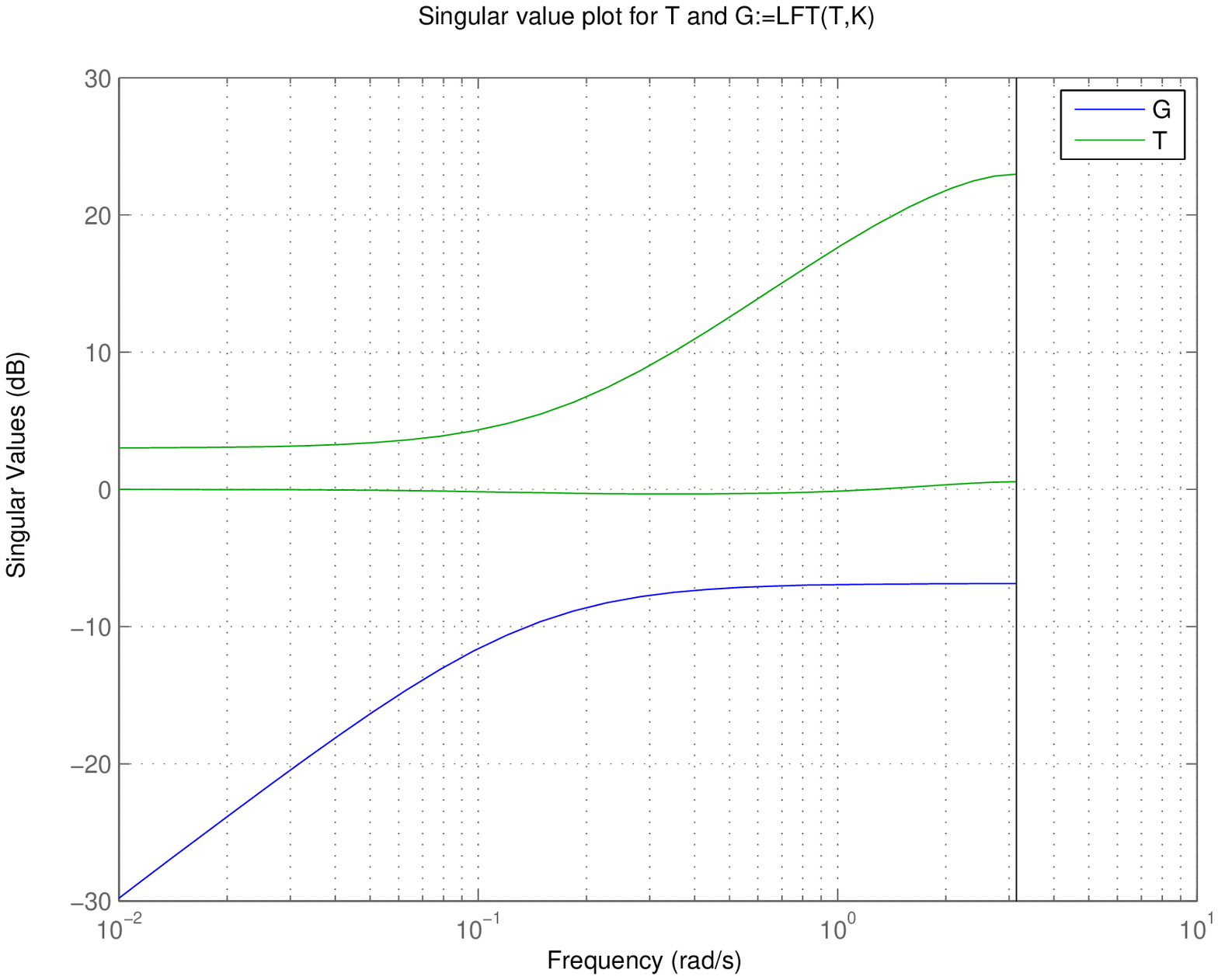}
\caption{}
\label{sigma_plot}
\end{figure}

\section{CONCLUSIONS}
We provided in this paper sufficient conditions for the existence of suboptimal $\hf$ controllers, considering a general discrete--time system. A realization--based characterization for the class of all stabilizing and contracting controllers was given. Our formulas are simple and numerically reliable for real--time applications, as it was shown in Section V. Necessary conditions and the separation structure of the $\hf$ controller will be investigated in a future work.




\section*{APPENDIX}

\begin{figure*}
\bel{1blocctrl}
\textbf{C}_1(z) =\ba{c|cc}
A-zE - (B_2D_{12}^{-1}C_1+B_1D_{21}^{-1} C_2)(\al-\beta z) & B_1D_{21}^{-1} & B_2D_{12}^{-1}\\
\hline
-D_{12}^{-1}C_1 & 0 & D_{12}^{-1}\\
-D_{21}^{-1}C_2 & D_{21}^{-1} & 0
\ea_{z_0}.
\ee
\bel{2blocctrl}
\textbf{C}_2(z)=\ba{c|cc}
A-zE + (B_2F_2-B_1D_{21}^{-1}C_2)(\al-\beta z) & B_1D_{21}^{-1} & B_2(D_{12}^*D_{12})^{-\frac{1}{2}}\\
\hline
F_2 & 0 & (D_{12}^*D_{12})^{-\frac{1}{2}}\\
-D_{21}^{-1}C_2-B_1^*X(\al E-\beta A) & D_{21}^{-1} & 0
\ea_{z_0}.
\ee
{\small\bel{tito}
\T_I(z)=\ba{c|cc}
A-zE+B_2F_2(\al-\beta z) & B_1 & B_2 (D_{12}^*D_{12})^{-\frac{1}{2}}\\
\hline
C_1+D_{12}F_2 & 0 & D_{12}(D_{12}^*D_{12})^{-\frac{1}{2}}\\
-F_1 & I & 0
\ea_{z_0},\T_O(z)=
\ba{c|cc}
A-zE+B_1F_1(\al-\beta z) & B_1 & B_2\\
\hline
-(D_{12}^*D_{12})^{\frac{1}{2}}F_2 & 0 & (D_{12}^*D_{12})^{\frac{1}{2}}\\
C_2+D_{21}F_1 & D_{21} & 0
\ea_{z_0}.
\ee}
\end{figure*}

\begin{figure*}
\bel{sigo}
\Sig_o:=\left(A^*-zE^*, \ba{cc}C_1^*& C^*_2\ea;\ B_1B^*_1, \ba{cc}0& B_1D_{21}^*\ea, \ba{cc}
-I_{p_1}& 0\\
0 & D_{21}D_{21}^*\ea\right), 
\ee

\bel{ctrl_2b}
\textbf{C}_3(z) = \ba{c|cc}
A-zE + (H_2C_2-B_2D_{12}^{-1}C_1)(\al-\beta z) & 	H_2  & 	-B_2D_{12}^{-1}-(\al E-\beta A)YC_1^*\\
\hline
-D_{12}^{-1}C_1 & 0 & D_{12}^{-1}\\
(D_{21}D_{21}^*)^{-\frac{1}{2}}C_2 & (D_{21}D_{21}^*)^{-\frac{1}{2}} & 0
\ea_{z_0},
\ee
$$
\text{where } H_2 = -(B_1D_{21}^*+(\al E-\beta A)YC_2^*)(D_{21}D_{21}^*)^{-1}.
$$

\end{figure*}

In order to proceed with the proofs, we need an additional result, for which the proof is omitted (for brevity).
\begin{lem}
\label{MIn}
Let $(C,A-zE)$ be a detectable pair and assume that there exists a matrix $X=X^*$ such that the following Stein equation holds: $E^*XE-A^*XA+C^*C=0.$ Then $X\le0$ if and only if $\Lam(A-zE)\subset\Di$.
\end{lem}

\begin{proof} {\bf (Proposition \ref{pp2})} 
$(i)\Rightarrow(ii)$: If $\fP(\ejt)<0,\forall\theta\in[0,2\pi)$, then $\fP$ has no zeros on $\pDi$. Thus $\mathcal{S}_{\fP}$, i.e., the symplectic pencil, has no generalized eigenvalues on $\pDi$, which implies that DDTARE($\Sig$) has a stabilizing solution, see \cite{oara2013ricc}. Further, since $z_0\in\pDi$, $\fP(z_0)=R<0$.

$(ii)\Rightarrow(i)$: Let $F$ be the stabilizing feedback as in \eqref{stabfeed}. Consider the {\it spectral factor}
$\textbf{S}(z):=
\ba{c|c}
A-zE & B\\ \hline
 -F& I
\ea_{z_0}.$
It can be easily checked that the factorization $\fP(z)=\textbf{S}^\#(z)R\textbf{S}(z)$ holds. Moreover, $\textbf{S}\in\rhf$ and $\textbf{S}^{-1}\in\rhf$, since $\Lam(A-zE+BF(\al-\beta z))\subset\Di$. Thus $\textbf{S}$ is a unity in $\rhf$. Since $R<0$, $\fP(\ejt)<0,\forall \theta\in[0,2\pi)$.
\end{proof}

\begin{proof} {\bf (Proposition \ref{BRL})}
$(i)\Rightarrow(ii)$: Note that 
$$\|\G\|_\infty<1 \Leftrightarrow \G^\#(\ejt)\G(\ejt)-I<0,\forall\theta\in[0,2\pi).$$ 
After manipulations we get that $\G^\#(z)\G(z)-I={\bf\Pi_{\tilde\Sig}}(z)$. Thus ${\bf\Pi_{\tilde\Sig}}(\ejt)<0,\forall\theta$. Since $A-zE$ is stable, it follows with Proposition \ref{pp2} that $D^*D-I<0$ and DDTARE($\tilde\Sig$) has a stabilizing solution $X=X^*$. It remains to prove that $X\le0$. It is easy to check that the DDTARE($\tilde\Sig$) has a stabilizing solution $X=X^*$ iff the following system of matrix equations 
\bel{kspys}
\baa{rcl}
D^*D-I&=&-V^*V\\
(\al E-\beta A)^*XB+C^*D&=&-W^*V\\
E*XE-A^*XA+C^*C &=& -W^*W
\eaa\ee
has a solution $(X=X^*,V,W)$, with $F=-V^{-1}W.$ Further, note that the last equation in \eqref{kspys} can be written as 
\bel{eech}
E^*XE-A^*XA+\ba{c}C\\W\ea^*\ba{c}C\\W\ea=0.
\ee
The pair $\left(\ba{c}C\\W\ea,A-zE\right)$ is detectable, since the pair $(W,A-zE)$ is detectable, from the fact that $A-zE-V^{-1}W(\al-\beta z)$ is stable. Using these conclusions, it follows from Lemma \ref{MIn} that $X\le0$.

$(ii)\Rightarrow(i)$: Following a similar reasoning as above, we have from $(ii)$ that $\left(\ba{c}C\\W\ea,A-zE\right)$ is detectable. Since $X\le0$ and the equality \eqref{eech} holds, we get from Lemma \ref{MIn} that $\Lam(A-zE)\subset\Di$. Using the implication $(ii)\Rightarrow(i)$ in Proposition \ref{pp2}, we have that ${\bf\Pi_{\tilde\Sig}}(\ejt)<0,\forall\theta$. But this is equivalent with $\|\G\|_\infty<1$. Thus $\G\in\bhf$.
\end{proof}
\begin{proof} {\bf (Theorem \ref{redh})} {\it If:} Let $$\K(z)=
\ba{c|c} A_K-zE_K& B_K \\
\hline
C_K & D_K
\ea_{z_0}
$$ be a minimal realization. Since $\K\in\bhf$, we have from Proposition \ref{BRL} that $D_K^*D_K-I<0$ and DDTARE($\Sig_K$) has a stabilizing solution $X_K=X_K^*\le0$, where $\Sig_K:=(A_K-zE_K,B_K;\ C_K^*C_K,C_K^*D_K, D_K^*D_K-I)$. Further, from $\T$ inner we get from Lemma \ref{inner_lem} that $D^*D=I$ and there is $X=X^*\le0$ such that \eqref{inner_e} holds. Compute now a minimal centered realization for $\G:=\textbf{LFT}(\T,\K)$, see Section 2.3.2 in \cite{eu13diss}. After leghty but simple algebraic manipulations we get that the realization of $\G$ satisfies condition (ii) in Proposition \ref{BRL}, with $X_G := \ba{cc}X&0\\0&X_K\ea=X_G^*\le0$, and $R_G:=D_{21}^*(D_K^*D_K-I)D_{21}<0$. It follows that $\G\in\bhf$.

{\it Only if:} From $(C_2,A-zE)$ detectable, $\T$ unitary, and Lemma \ref{MIn}, it follows that $\Lam(A-zE)\subset\Di$, thus $\T$ is inner. Since $\G\in\bhf$, $\|\G\|_\infty<1$, which is equivalent with $\G^\#(z)\G(z)-I<0$, for all $z\in\pDi$. Using equation \eqref{lft} and the fact that $\T_{21}$ is a unity in $\rhf$ (unimodular), we get after some  manipulations that $\K^\#(z)\K(z)-I<0$, for all $z\in\pDi$, which is equivalent with $\|\K\|_\infty<1$. The stability of $\K$ is a direct consequence of the fact that {\bf (H$_1$)} is fulfilled, that $\G$ is stable, and that $\T$ is inner. 
\end{proof}

We proceed now with the proof of our main result (stated in Theorem \ref{mainthm}), which is based on a successive reduction to simpler problems, called the one--block problem and the two--block problem. We borrowed the terminology from the model matching problem. 

Consider the {\it one--block problem}, for which $p_1 = m_2$, $p_2 = m_1$, i.e., $D_{12}$ and $D_{21}$ are square, and $\T_{12}$ and $\T_{21}$ are invertible, having only stable zeros, i.e.,
\begin{enumerate}[{\bf ({A}$_1$)}]
\item $D_{12}\in \C^{m_2\x m_2}$ is invertible and 
$\Lambda\big(A-zE-B_2D_{12}^{-1}C_1(\al-\beta z)\big)\subset\Di.$
\item $D_{21}\in \C^{m_1\x m_1}$ is invertible and 
$\Lambda\big(A-zE-B_1D_{21}^{-1}C_2(\al-\beta z)\big)\subset\Di.$
\end{enumerate}

\begin{prop}
\label{1bloc}
For the one--block problem the class of all controllers that solve the $\hf$ control problem is ${\normalfont\K=\textbf{LFT}(\textbf{C}_1,\Q)}$, $\Q\in\bhf$ is arbitrary and $\textbf{C}_1$ is in \eqref{1blocctrl}. 
\end{prop}
\begin{proof} Let $\T_R=\T\otimes \textbf{C}_1$. With $\textbf{C}_1$ from \eqref{1blocctrl} we get after an equivalence transformation that $\T_R=\ba{cc}0 & I \\ I & 0\ea$. Thus $\G={\bf LFT}(\T_R,\Q)=\Q\in\bhf$. 

Conversely, let $\K$ be such that $\G\in\bhf$. Take $\G\equiv\Q\in\bhf$ be an arbitrary but fixed parameter. Then 
$\textbf{LFT}(\textbf{C}_1,\Q)=\textbf{LFT}(\textbf{C}_1,\G)=\textbf{LFT}(\textbf{C}_1,\textbf{LFT}(\T,\K)) =
\textbf{LFT}(\textbf{C}_1\otimes\T,\K).$ It can be checked that in this case $\textbf{C}_1\otimes\T=\T\otimes \textbf{C}_1=\T_R$ (this is not generally true). Thus $\textbf{LFT}(\textbf{C}_1,\Q)=\K$. 
\end{proof}

Consider now the {\it two--block problem}, for which $p_2=m_1$, and the hypotheses {\bf (H$_2$)} and {\bf (A$_2$)} are fulfilled.  Let $\Sig_c$ be as in Box \ref{mainres}. 
\begin{prop}
\label{2bloc}
Assume that DDTARE($\Sig_c$) has a stabilizing solution $X=X^*\le0$. Then the two--block problem has a solution. Moreover, the class of all controllers is ${\normalfont\K=\textbf{LFT}(\textbf{C}_2,\Q)}$, with $\Q\in\bhf$, and $\textbf{C}_2$ is given in \eqref{2blocctrl}.
\end{prop}
\begin{proof}
Let $F_c$ be the stabilizing feedback, see Box \ref{mainres}. Consider the systems $\T_I$ and $\T_O$ in \eqref{tito}. After manipulations,  we obtain that $\T=\T_I\otimes\T_O$. Moreover, $\T_I$ is inner, since the realization \eqref{tito} satisfies the equations given in Lemma \ref{inner_lem}, with $X=X^*\le0$ the stabilizing solution of the DDTARE($\Sig_c$). Also, it can be easily checked that $\T_I$ satisfies the hypotheses of Theorem \ref{redh}. 


We claim that ${\bf LFT}(\T,\K)\in\bhf\Leftrightarrow {\bf LFT}(\T_O,\K)\in\bhf.$ Here follows the proof. 
${\bf LFT}(\T,\K)={\bf LFT}(\T_I\otimes\T_O,\K)={\bf LFT}(\T_I,{\bf LFT}(\T_O,\K))\in\bhf$. It follows from Theorem \ref{redh} that ${\bf LFT}(\T_O,\K)\in\bhf$. Conversely, let ${\bf LFT}(\T_O,\K)\in\bhf$ be a controller for the inner system $\T_I$. Then, we have from Theorem \ref{redh} that ${\bf LFT}(\T_I,{\bf LFT}(\T_O,\K))\in\bhf$. But this is equivalent with  ${\bf LFT}(\T,\K)\in\bhf$, since $\T_I\otimes\T_O=\T$. The claim is completely proven.

Therefore, it is enough to find the the class of controllers for $\T_O$. Further, it is easy to show that $\T_O$ in \eqref{tito} satisfies the assumptions {\bf ({A}$_1$)} and {\bf ({A}$_2$)} for the one--block problem. Compute {\bf C}$_1$ in \eqref{1blocctrl} for $\T_O$ to get {\bf C}$_2$ in \eqref{2blocctrl}. 
\end{proof}

The next result follows by duality from Proposition \ref{2bloc}. Consider $\Sig_o$ given in \eqref{sigo}. 
\begin{prop}
\label{2blocdual}
Assume $p_1=m_2$,  {\bf (A$_1$)},  {\bf (H$_3$)}, and that DDTARE($\Sig_o$) has a stabilizing solution $Y=Y^*\le0$. Then the dual two--block problem has a solution. Moreover, the class of all controllers is ${\normalfont\K=\textbf{LFT}(\textbf{C}_3,\Q)}$, where $\Q\in\bhf$ is arbitrary, and $\textbf{C}_3$ is given in \eqref{ctrl_2b}.
\end{prop}

\begin{proof} {\bf (Theorem \ref{mainthm})} We assume that {\bf (H$_1$)}, {\bf (H$_2$)}, and {\bf (H$_3$)} hold. Suppose that DDTARE($\Sig_c$) has a stabilizing solution $X=X^*\le0$. Consider now the systems $\T_I$ and $\T_O$, given in \eqref{tito}. We have shown that it is enough to find the the class of controllers for $\T_O$. It is easy to check that, in this case, $\T_O$ satisfies {\bf (A$_2$)}. Write now $\Sig_o$ in \eqref{sigo} and DDTARE($\Sig_o$) for $\T_O$ to obtain $\Sig_\x$ in Box \ref{mainres} and DDTARE($\Sig_\x$). Further, assume that DDTARE($\Sig_\x$) has a stabilizing solution $Z=Z^*\le0$. Therefore, $\T_O$ satisfies the assumptions in Proposition \ref{2blocdual}. The parametrization of all controllers that solve the $\hf$ control problem in Theorem \ref{mainthm} is now a consequence of Proposition \ref{2blocdual} and some straightforward manipulations. This completes whole the proof. 
\end{proof}




\vspace{-5mm}
\bibliographystyle{IEEEtran}
\bibliography{bib.bib}

\end{document}